\documentclass{amsart}
\usepackage{amssymb,amsfonts}
\usepackage[all,arc]{xy}
\usepackage{enumerate}
\usepackage{color,soul}
    \sethlcolor{white}
\usepackage{mathrsfs}
\usepackage{pinlabel}
\usepackage{anyfontsize}
\usepackage{t1enc}
\usepackage{rotating}
\usepackage{xparse}
\usepackage{graphicx}
\usepackage{caption}
\usepackage{subcaption}
\usepackage{pinlabel}
\usepackage{chngcntr}
\usepackage{fancyhdr}
\usepackage{tikz,tikz-cd,float}

\newtheorem{thm}{Theorem}[section]

\newtheorem{cor}[thm]{Corollary}
\newtheorem{prop}[thm]{Proposition}
\newtheorem{lem}[thm]{Lemma}

\usepackage{graphicx}
\theoremstyle{definition}
\newtheorem{defn}[thm]{Definition}

\newtheorem{exmp}[thm]{Example}

\newtheorem{notn}[thm]{Notation}

\theoremstyle{remark}
\newtheorem{rem}[thm]{Remark}

\usepackage{stmaryrd}

\DeclareMathOperator{\Hom}{Hom}

\DeclareMathOperator{\End}{End}
\DeclareMathOperator{\Aut}{Aut}
\newcommand{\tr}{\triangleright}

\newcommand{\TB}{\mathcal{TB}}
\newcommand{\TT}{\mathcal{T}}
\newcommand{\QQ}{\mathcal{Q}}
\newcommand{\LL}{\mathcal{L}}
\newcommand{\dih}{\mathbb{Z}_n^{dih}}

\newcommand{\dihprime}{\mathbb{Z}_{p^\alpha}^{dih}}

\providecommand{\keywords}[1]
{
  \small	
  \textbf{\textit{Keywords---}} #1
}
\makeatletter
\let\c@equation\c@thm
\makeatother
\numberwithin{equation}{section}

\title{Quandle coloring quivers and 2-bridge links}

\author[T. Khandhawit]{Tirasan Khandhawit}
\address{Department of Mathematics, Faculty of Science, Mahidol University\newline Centre of Excellence in Mathematics, CHE, Thailand}
\email{tirasan.kha@mahidol.ac.th}

\author[K. Kruaykitanon]{Korn Kruaykitanon}
\address{Department of Mathematics, Faculty of Science, Mahidol University, Thailand}
\email{korn.kru@student.mahidol.edu}

\author[P. Pongtanapaisan]{Puttipong Pongtanapaisan}
\address{School of Mathematical and Statistical Sciences, Arizona State University}
\email{ppongtan@asu.edu}

\begin{document}

 \begin{abstract}
The quandle coloring quiver was introduced by Cho and Nelson as a categorification of the quandle coloring number. In some cases, it has been shown that the quiver invariant offers more information than other quandle enhancements. In this paper, we compute the quandle coloring quivers of 2-bridge links with respect to the dihedral quandles.
\end{abstract}
\keywords{Quandles, 2-bridge links, quivers}

\maketitle
\section{Introduction}
A \textit{quandle} is an algebraic structure whose axioms are inspired by the Reidemeister moves on link diagrams \cite{joyce1982classifying,matveev1984distributive}. There is a natural quandle $Q(\LL)$ associated to each link $\LL$ called the \textit{fundamental quandle}, which gives rise to an invariant of the link. In fact, $Q(\LL)$ is a complete invariant when the link has one component \cite{joyce1982classifying,matveev1984distributive}. Studying presentations of $Q(\LL)$ can be difficult, and therefore, it is common to extract some information by considering the set of homomorphisms from $Q(\LL)$ to a different quandle $X.$ The cardinality of such a set $|\operatorname{Hom}(Q(\LL),X)|$ is often called the \textit{quandle coloring number}, which has been investigated by many quandle theorists over the years.

Since a set contains more information in addition to its cardinality, the quandle coloring number can be enhanced to give a stronger link invariant. For more details on some examples of useful enhancements such as cocycle and module enhancements, the readers are encouraged to consult \cite{elhamdadi2015quandles}. This paper concerns a particular enhancement introduced by Cho and Nelson called the \textit{quandle coloring quiver} $\mathcal{Q}(\LL)$ \cite{cho2019quandle}. Roughly, elements of $\operatorname{Hom}(Q(\LL),X)$ can be thought of as vertices scattered all over the place, where each vertex represents an assignment of a coloring to $\LL$. The quiver-valued invariant $\mathcal{Q}(\LL)$ gives a way to organize these vertices into a directed graph.

For some particular choices of target quandles $X$ appearing in $\operatorname{Hom}(Q(\LL),X)$, the quandle coloring quivers have been determined for various families of links \cite{basi2021quandle,zhou2023quandle,cazet2023quandles}. It has also been shown that in some cases, the quiver gives more information than cocycle and module enhancements \cite{nelson2019quandle,istanbouli2020quandle}. In this paper, we calculate the quandle quivers for all 2-bridge links with respect to any choice of dihedral quandle. This is particularly interesting when we use the dihedral quandle $\mathbb{Z}_n^{dih}$ of composite order $n$ since the quandle coloring quiver is determined by the coloring number when $n=p_1p_2\cdots p_k$, where $p_i$ is prime \cite{taniguchi2021quandle}. To demonstrate this, we give some more examples where our computations offer more information than the quandle counting invariants in the final section.

\subsection*{Organization}
This paper is organized as follows. In Section \ref{sec:prelims}, we discuss basic definitions from quandle theory and knot theory. In Section \ref{section:coloringnumber}, we calculate the quandle coloring number of 2-bridge links. The coloring number is needed as it is the number of vertices of the quandle coloring quiver invariant. In Section \ref{sec:quiverrationallink}, we prove our main result. Before stating the result in full generality, we discuss the case when $n$ is a power of a prime for ease of reading. We end the paper with more examples where our quiver computations give proper quandle enhancements. 
\section{Preliminaries}
\label{sec:prelims}
In this section, we review some relevant terminologies.
\subsection{Quandles}
\begin{defn} A \textit{quandle} is a nonempty set $X$ equipped with a binary operation $\tr:X\times X\to X$ such that the following properties hold:\begin{itemize}
\item [\textbf{Q1}:] $x\tr x=x$ for all $x\in X$.
\item [\textbf{Q2}:] The map $\beta_y : X\to X$, given by $\beta_y(x)= x\tr y$, is invertible for all $y\in X$.
\item [\textbf{Q3}:] $(x\tr y )\tr z = (x\tr z) \tr (y\tr z)$ for all $x,y,z\in X$.
\end{itemize}\end{defn}
Since $\beta_y$ is invertible, we have a bijection $\beta_y^{-1}:X\to X$. Define $\tr^{-1}: X\times X \to X$ by $x \tr^{-1} y := \beta_y^{-1}(x)$. If $\beta_y=\beta_y^{-1}$ for all $y\in X$, or equivalently $\tr=\tr^{-1}$, then the quandle is said to be \textit{involutory}. In this paper, we primary work with dihedral quandles, which are in fact involutory quandles.

\begin{exmp} For each $n\in\mathbb{N}$, on $\mathbb{Z}_n=\{0,1,2,\dots,n-1\}$, $x\tr y := 2y-x\pmod{n}$ defines the \textit{dihedral quandle} of order $n$. Denote by $\dih$ the dihedral quandle of order $n$. For all $x,y\in \mathbb{Z}_n$, we have $\beta_y\circ\beta_y(x)\equiv \beta_y(2y-x)\equiv 2y-(2y-x)=x\pmod{n}$. From here we see that dihedral quandles are involutory. \end{exmp}
Often, it is useful to study maps between quandles that behave well with the quandle axioms.

\begin{defn} A \textit{quandle homomorphism} from $(X,\tr_X)$ to $(Y,\tr_Y)$ is a map $f:X\to Y$ such that $f(x \tr_X y)=f(x)\tr_Y f(y)$. Denote by $\Hom(X,Y)$ the set of all quandle homomorphisms from $X$ to $Y$. \end{defn}

\begin{defn} Let $X$ be a quandle. A \textit{quandle endomorphism} on $X$ is a quandle homomorphism from $X$ to itself. A \textit{quandle automorphism} on $X$ is a quandle endomorphism on $X$ that is also a bijection. Denote by $\End(X)$ the set of all quandle endomorphisms on $X$ and by $\Aut(X)$ the set of all quandle automorphisms on $X$. \end{defn}

\begin{rem} Under the usual composition, $\End(X)$ has monoid structure, whereas $\Aut(X)$ has group structure.\end{rem}

There is a particularly natural quandle that can be defined from a link diagram.

\begin{defn} Let $\LL$ be an oriented link and $D$ be an oriented diagram of $\LL$ with $n$ strands, $x_1,x_2,\dots,x_n$. The \textit{fundamental quandle} of $D$ is the quandle freely generated by $x_1,x_2,\dots,x_n$ with relations from each crossing as in Figure \ref{fig:fundquandle}. The \textit{fundamental quandle} of an oriented link $\LL$ is defined to be the fundamental quandle of an oriented diagram of $\LL$.\end{defn}

\begin{figure}[ht]
    \centering
    \includegraphics[width=0.3\textwidth]{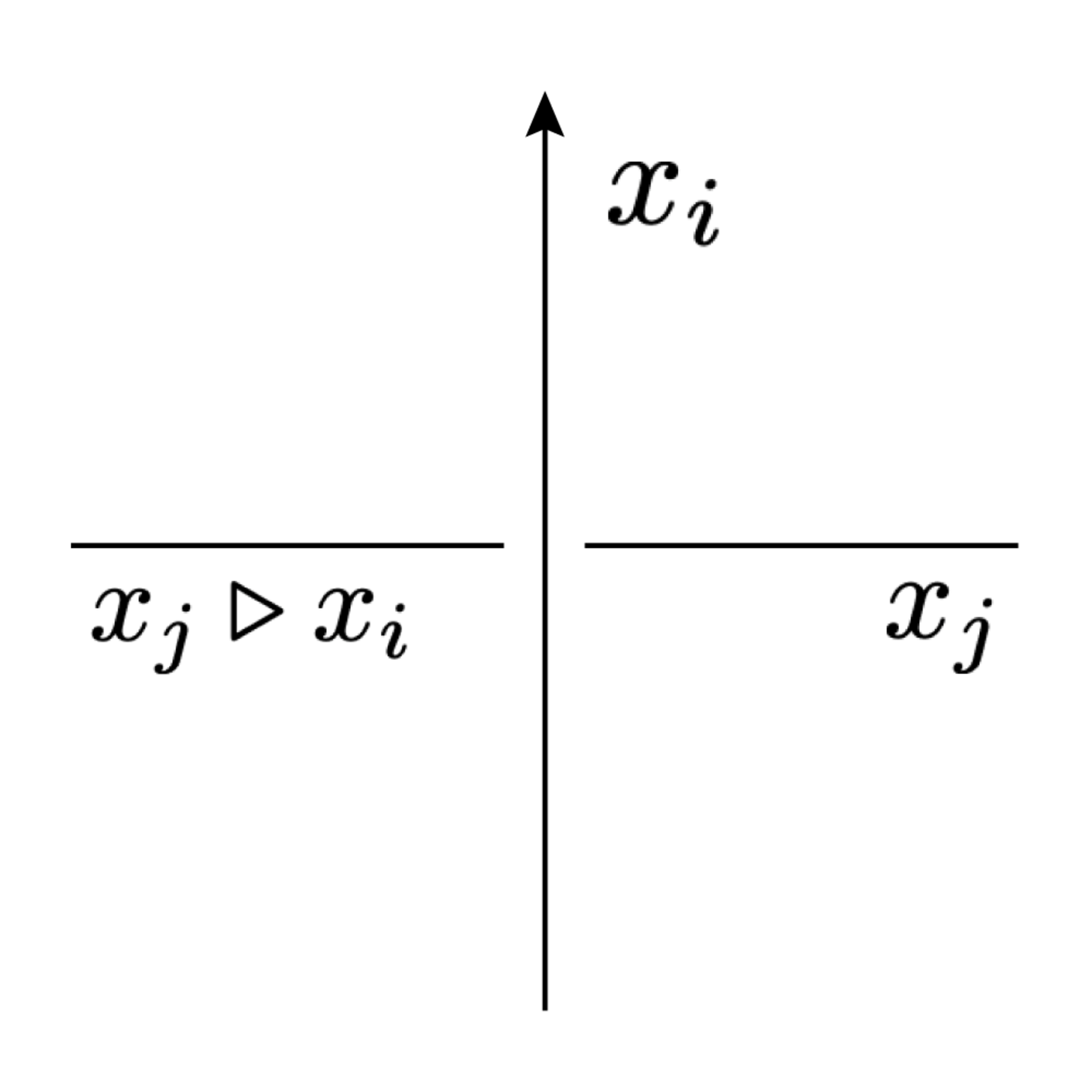}
    \caption{\label{fig:fundquandle} The relation for the fundamental quandle at each crossing}
\end{figure}

A basic way to study homomorphisms between quandles is to count how many there are.
\begin{defn} Let $\LL$ be a link and $X$ be a finite quandle. An \textit{$X$-quandle coloring} of $\LL$ is a quandle homomorphism $Q(\LL)\to X$. The \textit{$X$-quandle counting invariant} of the link $\LL$ is the number of $X$-quandle colorings of $\LL$, i.e. the size of the set $\Hom(Q(\LL),X)$. This cardinality is also called the \textit{quandle coloring number}. \end{defn}

\begin{exmp} For any link $\LL$ and a quandle $X$. Fix $x\in X$. Then, $\psi_x : Q(\LL)\to X$, given by $\psi_x(x_i):=x$, defines a quandle homomorphism since at each crossing we have $\psi_x(x_i\tr x_j)=x=x\tr x= \psi_x(x_i)\tr \psi_x(x_j)$. Such quandle colorings are called the \textit{trivial quandle coloring}. In general, we have $\{\psi_x:x\in X\}\subseteq\Hom(Q(\LL),X)$, so $|\Hom(Q(\LL),X)|\geq |X|$.\end{exmp}

Since the set of homomorphisms contains more information than its cardinality, various quandle enhancements have been defined. The following concept is particularly relevant to this paper.

\begin{defn}[See \cite{cho2019quandle}] Let $X$ be a finite quandle. Fix $S\subseteq \Hom(X,X)$. The \textit{$X$-quandle coloring quiver} $\QQ_X^S(\LL)$ of a link $\LL$ with respect to $S$ is the direct graph with vertex set $\Hom(Q(\LL),X)$ and directed edges $\psi_1 \overset{f}{\to} \psi_2$ whenever $\psi_2=f\circ \psi_1$ and $f\in S$. When $S=\Hom(X,X)$, we denote the corresponding quiver by simply $\QQ_X(\LL)$ and call it the \textit{full quandle coloring quiver}.\end{defn}

\begin{notn} Denote by $(\overleftrightarrow{K_n},\widehat{m})$ the directed graph with $n$ vertices where every vertex has $m$ directed edges from itself to each vertex. For each graph $G$ and $H$, define $G \overleftarrow{\nabla}_{\widehat{m}}H$ to be the disjoint union graph $G\sqcup H$ with additional $m$ directed edges from every vertex of $H$ to each vertex of $G$. \end{notn}

\subsection{Rational Tangles and Links}
An \textit{$n$-string tangle} is a collection of $n$ properly embedded disjoint arcs in the 3-ball. In this paper, we will work exclusively with 2-string tangles. Thus, we will simply refer to 2-string tangles as \textit{tangles} for brevity. A tangle can also be defined diagrammatically.

\begin{defn} 
A \textit{tangle diagram} is a portion of a link diagram surrounded by a circle intersecting the link diagram in four points labelled $NE,NW,SE,SW$. Two tangle diagrams are \textit{equivalent} if and only if one can be obtained from another by Reidemeister moves in finitely many steps inside the surrounding circle while the four points remain fixed.
\end{defn}
We will now give a definition of rational tangles. We note that there are other ways to define the equivalent object in the literature.

Let [0] denote the horizontal tangle shown in Figure \ref{fig:tangle1} (left). For an integer $p\neq 0$, let $[p]$ denote the tangle obtained from twisting the $NE$ and $SE$ endpoints $p$ times, where the sign is positive (resp. negative) if the overstrand has positive (resp. negative) slope (see Figure \ref{fig:tangle1}).
\begin{figure}[ht]
    \centering
    \includegraphics[width=0.7\textwidth]{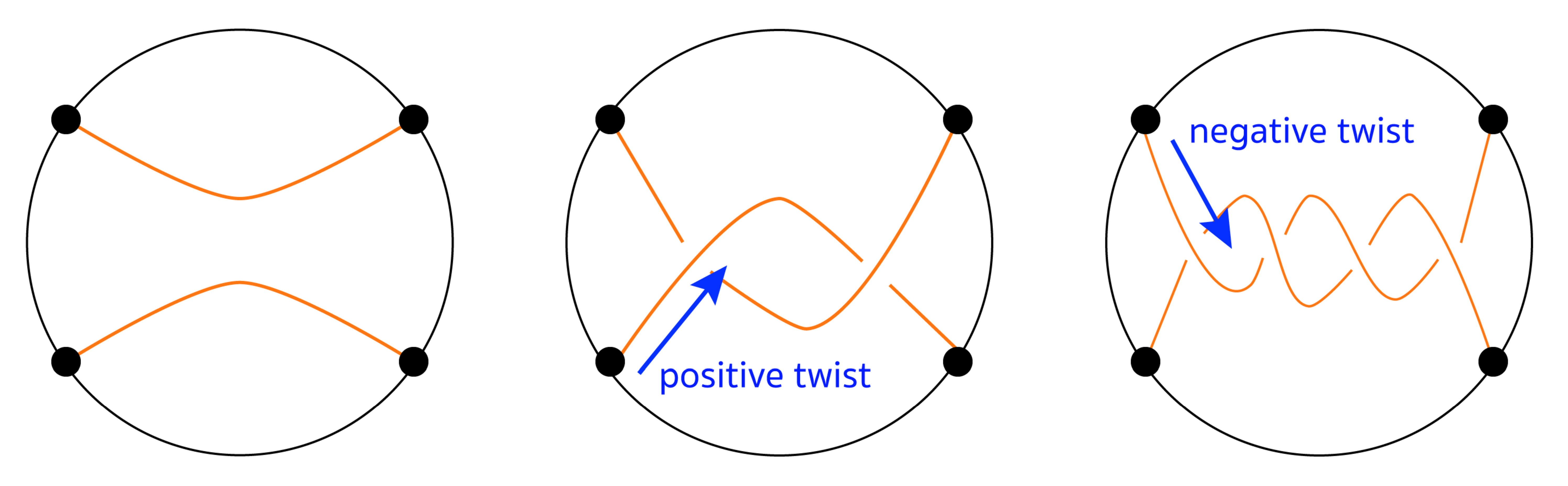}
    \caption{Left: The [0] tangle. Middle: The [+2] tangle. Right: The [-4] tangle.}
    \label{fig:tangle1}
\end{figure}
\begin{figure}[ht]
    \centering
    \includegraphics[width=1\textwidth]{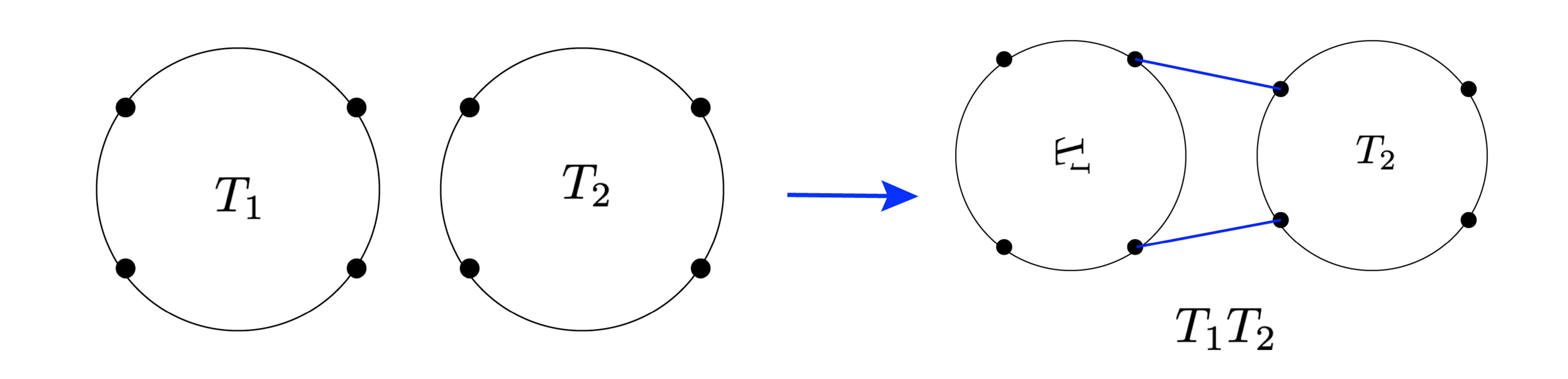}
    \caption{construction of $T_1T_2$ from $T_1,T_2$.}
    \label{fig:tangle2}
\end{figure}

Given two tangles $T_1$ and $T_2$, we can connect the two tangles into a new one. Let us denote by $T_1T_2$ the tangle obtained from reflecting $T_1$ along $NW-SE$ line and connecting it to $T_2$ from the left. (see Figure \ref{fig:tangle2}) Note that in general, $T_1T_2\neq T_2T_1$.

Let $N\geq 1$, and $p_1,p_2,\dots,p_N$ be integers. Let $[p_1p_2\dots p_N]$ be the tangle $T_N$, where $T_1=[p_1]$ and $T_j=T_{j-1}[p_j]$ for $1\leq j\leq N$. This kind of tangle is called a \textit{rational tangle}.

To each rational tangle $[p_1p_2\dots p_N]$, there is an associated rational number $$\overline{p_1p_2\dots p_N} := p_N+\frac{1}{\dots + \frac{1}{p_2+\frac{1}{p_1}}}$$ that is a complete tangle invariant. That is, Conway showed that two rational tangles are equivalent if and only if their rational numbers are equal \cite{conway1970enumeration}.

The numerator closure of a rational tangle yields a \textit{rational link}. It can be shown that rational links are precisely two-bridge links. Let us denote by $\TB(p_1p_2\dots p_N)$, or $\TB(\overline{p_1p_2\dots p_N})$ the closure of the rational tangle $[p_1p_2\dots p_N]$ (see Figure \ref{fig:TBeven}).

We note that any rational tangle can be put in a canonical form so that each $p_i$ in $\TB(p_1p_2\dots p_N)$ has the same sign \cite{kauffman2004classification}. Since $\TB((-p_1)(-p_2)\dots (-p_N))$ is the mirror image of $\TB(p_1p_2\dots p_N)$, their involutorized fundamental quandles are isomorphic. Hence, their quandle enhancements, e.g. coloring number, quiver, are isomorphic. From now on, we shall assume that $p_1,p_2,\dots,p_N>0$.

\begin{figure}[ht]
    \centering
    \includegraphics[width=1\textwidth]{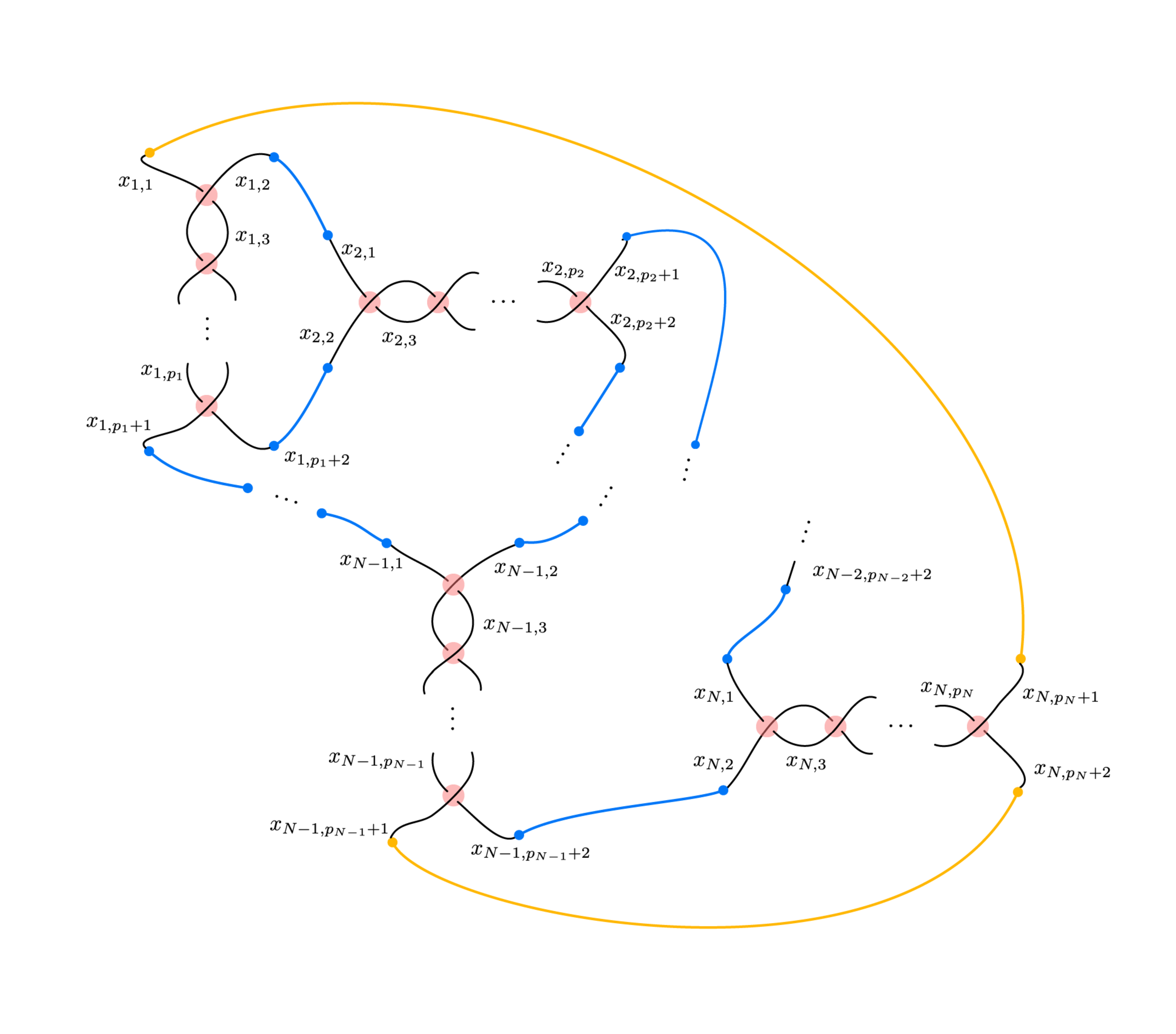}
    \caption{Labels for the strands of a 2-bridge link.}
    \label{fig:TBeven}
\end{figure}

\section{The quandle coloring numbers}\label{section:coloringnumber}
The main goal of this section is to determine the number of colorings of 2-bridge links by dihedral quandles. We begin by discussing a presentation for the fundamental quandle of 2-bridge links:

 \begin{align*}Q(\TB(p_1p_2\dots p_N))&=\langle x_{j,i} \text{ for } 1\leq j\leq N \text{ and } 1\leq i\leq p_j+2\mid \\ 
& x_{j,i} = x_{j,i-2} \tr^\pm x_{j,i-1} \text{ for } 1\leq j\leq N \text{ and } 3\leq i\leq p_j+2,\\
& x_{2,1}=x_{1,2}, x_{2,2}=x_{1,p_1+2},\\
& x_{j,1}=x_{j-2,p_{j-2}+1}, x_{j,2}=x_{j-1,p_{j-1}+2} \text{ for }3\leq j\leq N,\\
& x_{N,p_N+1}=x_{1,1} , x_{N,p_N+2}=x_{N-1,p_{N-1}+1}\rangle,\end{align*}

Since dihedral quandles are involutory, i.e. $\tr=\tr^{-1}$, for any quandle homomorphism $\psi:Q(\TB(p_1p_2\dots p_N))\to \dih$ we have the following relations \begin{align*}
& \psi(x_{j,i} )= \psi(x_{j,i-2}) \tr \psi(x_{j,i-1}) \text{ for } 1\leq j\leq N \text{ and } 3\leq i\leq p_j+2,\\
& \psi(x_{2,1})=\psi(x_{1,2}), \psi(x_{2,2})=\psi(x_{1,p_1+2}),\\
& \psi(x_{j,1})=\psi(x_{j-2,p_{j-2}+1}), \psi(x_{j,2})=\psi(x_{j-1,p_{j-1}+2}) \text{ for }3\leq j\leq N,\\
& \psi(x_{N,p_N+1})=\psi(x_{1,1} ), \psi(x_{N,p_N+2})=\psi(x_{N-1,p_{N-1}+1}).
\end{align*} Moreover, any map $\psi : \{x_{j,i}\mid  1\leq j\leq N \text{ and } 1\leq i\leq p_j+2\}  \to \dih$ satisfying the relations extends to a unique quandle homomorphism $\tilde{\psi}: Q(\TB(p_1p_2\dots p_N))\to \dih$.

Next, we prove an important proposition relating the colorings of two generating strands. This generalizes Proposition 2.4 of \cite{basi2021quandle}.

For a rational tangle $[p_1p_2\dots p_N]$, let $\Delta_j$ be the numerator of the rational number $\overline{p_1p_2\dots p_j}$ and also denote by $\Delta := \Delta_N$. Note that $\Delta_j$ satisfies recurrence relation
\begin{align*}
   \Delta_0 := 1, \, \Delta_1 = p_1, \, \Delta_j = p_j \Delta_{j-1} + \Delta_{j-2} .
\end{align*}
In fact, the number $\Delta$ is the determinant of $\TB(p_1p_2\dots p_N)$ (see \cite{kauffman2009rational}).

\begin{prop}\label{preunivtb} For $\psi \in \Hom(Q(\TB\left(p_1p_2\dots p_N\right)),\dih)$, we have $$\Delta\psi(x_{1,1})\equiv \Delta\psi(x_{1,2}) \pmod{n}.$$ \end{prop}
\begin{proof}
Since the term of the form $a\psi(x_{1,2})-(a-1)\psi(x_{1,1})$ appears frequently in the proofs, we let $[a]:=a\psi(x_{1,2})-(a-1)\psi(x_{1,1})$. We observe that for all $1\leq j\leq N$, we have $\psi(x_{j,p_j+1})=p_j\psi(x_{j,2})-(p_j-1)\psi(x_{j,1})$, and $\psi(x_{j,p_j+2})=(p_j+1)\psi(x_{j,2})-p_j\psi(x_{j,1})$.

\textbf{Claim:} For all $1\leq j\leq N$, we have $\psi(x_{j,p_j+1})=[\Delta_j]$ and $\psi(x_{j,p_j+2})=[\Delta_j+\Delta_{j-1}]$. 

\begin{proof}[Proof of Claim]
    We prove the claim by induction. For base case $j=1$, we note that $\psi(x_{1,p_1+1})=[p_1]=[\Delta_1]$, and $\psi(x_{1,p_1+2})=[p_1+1]=[\Delta_1+\Delta_0]$. Let $1\leq j\leq N$. Suppose that the claim hold true for all positive integer less than $j$.  \begin{description}
\item [Case 1] $j=2$. We have $\psi(x_{2,1})=\psi(x_{1,2})=[1]$ and $\psi(x_{2,2})=\psi(x_{1,p_1+2})=[p_1+1]$. This gives $\psi(x_{2,p_j+1})=p_2 [p_1+1]-(p_2-1)[1]=[p_2p_1+1]=[\Delta_2]$, and $\psi(x_{2,p_j+2})=(p_2+1)[p_1+1]-p_2[1]=[p_2p_1+p_1+1]=[\Delta_2+\Delta_1]$.
\item [Case 2] $j\geq 3$. As $j-1,j-2\geq 1$, we apply inductive hypothesis and obtain
\begin{align*}
\psi(x_{j,1}) &= \psi(x_{j-2,p_{j-2}+1}) = [\Delta_{j-2}],\\
\psi(x_{j,2}) &= \psi(x_{j-1,p_{j-1}+2}) = [\Delta_{j-1}+\Delta_{j-2}],\\ 
\psi(x_{j,p_j+1}) &= p_j[\Delta_{j-1}+\Delta_{j-2}] - (p_j-1) [\Delta_{j-2}] = [p_j\Delta_{j-1}+\Delta_{j-2}]=[\Delta_j],\\
\psi(x_{j,p_j+2}) &= (p_j+1)[\Delta_{j-1}+\Delta_{j-2}] - p_j [\Delta_{j-2}] \\
&= [p_j\Delta_{j-1}+\Delta_{j-2}+\Delta_{j-1}] = [\Delta_j+\Delta_{j-1}].
\end{align*} \end{description} 
Thus, the claim is verified.

\end{proof}
With the claim proved, we have $\psi(x_{N,p_N+1})=[\Delta_N]$ and $\psi(x_{N,p_N+2})=[\Delta_N + \Delta_{N-1}]$. The relations from the closure of the tangle give a single equation 
\begin{center}
    $\Delta_N \psi(x_{1,1}) \equiv \Delta_N \psi(x_{1,2}) \pmod{n}$.
\end{center}

Hence, the assertion is proved. 
\end{proof}

\begin{prop} \label{univtb} Any map $\psi: \{x_{1,1},x_{1,2}\} \to \dih$ such that $\Delta \psi(x_{1,1})\equiv \Delta \psi(x_{1,2})\pmod{n}$ extends to a unique quandle homomorphism $\tilde\psi : Q(\TB\left(p_1p_2\dots p_N\right)) \to \dih $, i.e. the diagram $$\begin{tikzcd} \{x_{1,1},x_{1,2}\} \arrow[r,"\psi"] \arrow[d,hook,"i"]& \dih \\ Q(\TB\left(p_1p_2\dots p_N\right))\arrow[ru,dashed,"\tilde{\psi}"swap] \end{tikzcd}$$ commutes. \end{prop}

\begin{proof} We extend $\psi$ to $\bar\psi: \{x_{j,i}\mid  1\leq j\leq N \text{ and } 1\leq i\leq p_j+2\}  \to \dih$ uniquely to other generators recursively using the following relations  \begin{align*}\bar\psi(x_{j,i} ) &= \bar\psi(x_{j,i-2}) \tr \bar\psi(x_{j,i-1}) \text{ for } 1\leq j\leq N \text{ and } 3\leq i\leq p_j+2,\\
\bar\psi(x_{2,1})&=\bar\psi(x_{1,2}), \bar\psi(x_{2,2})=\bar\psi(x_{1,p_1+2}),\\
 \bar\psi(x_{j,1})&=\bar\psi(x_{j-2,p_{j-2}+1}), \bar\psi(x_{j,2})=\bar\psi(x_{j-1,p_{j-1}+2}) \text{ for }3\leq j\leq N. \end{align*} From the proof of Proposition \ref{preunivtb}, we see that $\bar\psi(x_{N,p_N+1})=[\Delta_N]$ and $\bar\psi(x_{N,p_N+2})=[\Delta_N + \Delta_{N-1}]$. Hence, the relations  $$\bar\psi(x_{N,p_N+1})=\bar\psi(x_{1,1} ), \bar\psi(x_{N,p_N+2})=\bar\psi(x_{N-1,p_{N-1}+1})$$ hold and $\bar\psi$ extends to a unique quandle homomorphism $\tilde\psi : Q(\TB(p_1p_2\dots p_N))\to \dih$.  \end{proof}

\begin{cor} \label{corcolornumber} The quandle coloring number of a 2-bridge link is given by the formula $|\Hom(Q(\TB\left(p_1p_2\dots p_N\right)),\dih)|=n\gcd(\Delta,n),$ where $n$ of which are trivial quandle colorings.\end{cor}

\begin{proof} By Proposition \ref{preunivtb} and \ref{univtb}, $|\Hom(Q(\TB\left(p_1p_2\dots p_N\right)),\dih)|$ is equal to the number of choices of $(\psi(x_{1,1}),\psi(x_{1,2}))\in\dih\times \dih$ such that $\Delta\psi(x_{1,1})\equiv \Delta\psi(x_{1,2}) \pmod{n}$ which is exactly $n\gcd(\Delta,n)$. Among all the colorings, there are $n$ trivial colorings corresponding to choices $\psi(x_{1,1})=\psi(x_{1,2})\in\dih$.\end{proof}

\section{The quandle coloring quiver of 2-bridge links}\label{sec:quiverrationallink}
It will turn out that the quiver invariant can be organized based on how the automorphism group $\Aut(\dih)$ acts on the colorings.

Throughout this section, we consider a 2-bridge link $\TB\left(\frac{N}{M}\right)$, where $N$ and $M$ are positive and relatively prime.
\begin{notn}
By Proposition \ref{univtb}, we denote by $[a,b]$ the unique quandle homomorphism $\psi \in \Hom(Q\left(\TB\left(\frac{N}{M}\right)\right), \dih)$ such that $\psi(x_{1,1})=a$ and $\psi(x_{1,2})=b$.
\end{notn}
 Note that such $a$ and $b$ satisfy $n \mid N(b-a)$.
\begin{notn}
    Analogously, for $x,y \in \mathbb{Z}_n$, there is a unique endomorphism $f \in \End(\dih)$ such that $f(0) = x$ and $f(1) =y$. We denote such an endomorphism by $\llbracket x,y\rrbracket$
\end{notn}
 Observe that $f(a) = ay - (a-1)x = (y-x)a + x \pmod{n}$. Moreover, $\llbracket x,y\rrbracket$ is an automorphism precisely when $y-x\in \mathbb{Z}_n^\times$.

\begin{prop}\label{autactonend}  $\Aut(\dih)$ acts on $\Hom(Q\left(\TB\left(\frac{N}{M}\right)\right), \dih)$ by post-composition, i.e. $f\tr [a,b] :=f\circ [a,b]=[f(a),f(b)]$. 
\end{prop}

\begin{proof} For $[a,b]\in \Hom(Q\left(\TB\left(\frac{N}{M}\right)\right), \dih)$ and $f=\llbracket x,y\rrbracket \in \Aut(\dih)$, we see that $n\mid N(y-x)(b-a)=N(f(b)-f(a))$, i.e. $[f(a),f(b)]\in  \Hom(Q\left(\TB\left(\frac{N}{M}\right)\right), \dih)$. Since composition is associative and $1_{\dih}\in \Aut(\dih)$ fixes any $[a,b]$, we see that all the group action axioms are satisfied.\end{proof}
\begin{notn}Write $\psi\sim \phi$ if $\psi$ and $\phi$ lie in the same orbit under the action.\end{notn}
\begin{lem} \label{numedgeequal}For $\psi,\psi',\phi,\phi' \in \Hom(Q\left(\TB\left(\frac{N}{M}\right)\right),\dih)$ such that $\psi\sim \psi'$ and $\phi\sim \phi'$, we have $$|\{f\in \End(\dih): \phi=f\circ \psi \}|=|\{f\in \End(\dih): \phi'=f\circ \psi'\}|.$$  \end{lem}

\begin{proof} Since $\psi\sim \psi'$ and $\phi\sim \phi'$, there exist $g,h\in \Aut(\dih)$ such that $\psi'=g\circ \psi$ and $\phi'=h\circ \phi$. Define two maps $T: \{f\in \End(\dih): \phi=f\circ \psi \} \to \{f\in \End(\dih): \phi'=f\circ \psi' \}$ by $f\mapsto h\circ f\circ g^{-1}$, and $S: \{f\in \End(\dih): \phi'=f\circ \psi' \} \to \{f\in \End(\dih): \phi=f\circ \psi \}$ by $f\mapsto h^{-1}\circ f\circ g$. We see that $T$ and $S$ are inverse to each other. Hence, two sets are of the same size. \end{proof}

By translation, any orbit contains an element of the form $[0,a]$.
Consequently, it suffices to consider edges between them, i.e. $$|\{f\in \End(\dih): [0,b]=f\circ [0,a]\}|=|\{f\in \End(\dih): f(0)=0, f(a)=b\}|.$$

\begin{lem} For $a,b\in \dih$, we have $$|\{f\in \End(\dih): f(0)=0, f(a)=b\}|=|\{x\in \mathbb{Z}_n: ax\equiv b \pmod{n}\}|.$$\end{lem}

\begin{proof} Two maps $f\mapsto f(1)$ and $x\mapsto \llbracket 0,x\rrbracket$ are inverses.
\end{proof}

It is a basic number theory result that $$\{x\in \mathbb{Z}_n: ax\equiv b \pmod{n}\}=\begin{cases} \gcd(a,n) &\text{ if }\gcd(a,n)\mid b,\\ 0 &\text{ else}.\end{cases}.$$ We immediately have our result.

\begin{prop} \label{countendo} For $a,b\in \dih$, we have $$|\{f\in \End(\dih): [0,b]=f\circ [0,a]\}|=\begin{cases} \gcd(a,n) &\text{ if }\gcd(a,n)\mid b,\\ 0 &\text{ else}.\end{cases}$$ \end{prop}

\subsection{The quiver when $n$ is a power of a prime}
Let us first consider the case when  $n=p^{\alpha}$, where $p$ a prime and $\alpha$ is a positive integer. 

\begin{defn}
    The \textit{$p$-adic valuation} of an integer $m$, denoted by $\nu_p(m)$, is the highest power of $p$ dividing $m$.
\end{defn}

Given $p, \alpha,$ and $N$, we set $\beta=\min\{\alpha,\nu_p(N)\}$. We now characterize orbits of 
$\Hom(Q\left(\TB\left(\frac{N}{M}\right)\right), \dihprime)$ and count endomorphisms between them.

\begin{lem} \label{orbitdecomp} 
Under the action of $\Aut(\dihprime)$ on $\Hom(Q\left(\TB\left(\frac{N}{M}\right)\right), \dihprime)$,  for $\alpha-\beta \leq j,j'\leq \alpha$ , we have
\begin{enumerate}
\item $[0,p^j]\in \Hom(Q\left(\TB\left(\frac{N}{M}\right)\right), \dihprime)$.
\item $[0,p^j]$ and $[0,p^{j'}]$ lie in same orbit if and only if $j=j'$.
\item The size of the orbit of $[0,p^j]$, denoted by $n_j$, is given by $$n_j=\begin{cases} p^{2\alpha-j-1}(p-1) &\text{ if } j<\alpha, \\ p^\alpha &\text{ if }j=\alpha. \end{cases}$$
\item $\Hom(Q\left(\TB\left(\frac{N}{M}\right)\right), \dihprime)$ is partitioned into orbits with $\{[0,p^j]: \alpha-\beta \leq j\leq \alpha\}$ being a complete set of representatives.
\item The number of endomorphisms of $\mathbb{Z}_{p^\alpha}^{dih}$  sending $[0, p^j]$ to $[0, p^{j'}]$, denoted by $n_{j,j'}$, is given by $$n_{j,j'}=\begin{cases} 0 & \text{ if } j>j',\\ p^j & \text{ if }j\leq j'.\end{cases}$$
\end{enumerate}
\end{lem}

\begin{proof} \begin{enumerate}
\item Since $j \ge  \alpha-\beta \ge \alpha-\nu_p(N)$, we have $p^\alpha \mid Np^j$.

\item The converse is obvious.  Without loss of generality, let us suppose that $j>j'$. 
We see that $\gcd(p^j,p^\alpha) = p^j \nmid p^{j'}$, so there is no automorphism from $[0,p^j]$ to $[0,p^{j'}]$ by Proposition~\ref{countendo}.

\item We first determine the size of stabilizer of $[0,p^j]$, which is equal to the number of $ \llbracket x,y\rrbracket \in\Aut(\dihprime)$ such that $\llbracket x,y\rrbracket \tr [0,p^j]=[0,p^j]$. Note that the size of $\Aut(\dihprime)$ is equal to $|\mathbb{Z}_{p^\alpha}| \cdot |\mathbb{Z}_{p^\alpha}^\times| = p^\alpha \phi(p^\alpha) = p^{2\alpha-1}(p-1)$.
\begin{description}
\item [Case 1] $j=\alpha$. In this case, it is equivalent to count a number of $\llbracket x,y\rrbracket$ such that $x=0$ and $y\in\mathbb{Z}_{p^\alpha}^\times$, so the stabilizer of $[0,p^\alpha] = [0,0]$ is of the size $|\mathbb{Z}_{p^\alpha}^\times|=\phi(p^\alpha)$.  By orbit-stabilizer theorem, the size of the orbit of $[0,p^\alpha]$ is $\frac{p^\alpha \phi(p^\alpha)}{\phi(p^\alpha)}=p^{\alpha}$.

\item [Case 2] $j<\alpha$. In this case, we count a number of $\llbracket x,y\rrbracket$ such that $x=0$ and $y p^j = p^j \pmod{p^\alpha}$. The last condition is equivalent to $y = 1+ kp^{\alpha-j}$ for a nonnegative integer $ k<p^j$. Hence, the stabilizer of $[0,p^j]$ is of the size $p^j$. By orbit-stabilizer theorem, the size of the orbit of $[0,p^j]$ is $\frac{p^\alpha \phi(p^\alpha)}{p^j}=p^{2\alpha-j-1}(p-1)$.
\end{description}

\item Consider the total size of the orbit of  $[0,p^j]$ for all $\alpha-\beta \leq j\leq \alpha$
\begin{align*}
\sum\limits_{\alpha-\beta \leq j \leq \alpha} n_j &= p^{\alpha} + \sum_{\alpha-\beta \leq j < \alpha}p^{2\alpha-j-1}(p-1) \\
&= p^{\alpha} + p^{2\alpha-1}(p-1)\cdot \frac{1}{p^{\alpha-\beta}} \sum_{0\leq j< \beta} \frac{1}{p^{j}}\\
&=p^{\alpha} + p^{2\alpha-1}(p-1)\cdot \frac{1}{p^{\alpha-\beta}} \cdot\frac{1-\frac{1}{p^{\beta}}}{1-\frac{1}{p}}\\
&= p^{\alpha+\beta}.
\end{align*}
One the other hand, we have $|\Hom(Q\left(\TB\left(\frac{N}{M}\right)\right), \dihprime)|=p^\alpha\gcd(N,p^\alpha)=p^{\alpha+\beta}$ by corollary~\ref{corcolornumber}. Hence  $[0,p^j]$ for $\alpha-\beta \leq j\leq \alpha$ are complete representatives.

\item  This follows from Proposition \ref{countendo}.
\end{enumerate}\end{proof}

Combining all the results from Lemma \ref{numedgeequal} and Lemma \ref{orbitdecomp}, we are able to determine the full quandle coloring quiver of the two-bridge link $\TB\left(\frac{N}{M}\right)$ with respect to the quandle $\dihprime$.

\begin{thm} \label{mainresult} Let $p$ be a prime, $\alpha\geq 1$ be an integer, and $N,M\in\mathbb{N}$ with $\gcd(N,M)=1$. The full coloring quiver of the two-bridge link $\TB\left(\frac{N}{M}\right)$ with respect to the quandle $\dihprime$ is given by $$\QQ_{\dihprime}\left(\TB\left(\frac{N}{M}\right)\right)\cong G_\beta,$$ where $\beta={\min\{\nu_p(N),\alpha\}}$, $G_0:= (\overleftrightarrow{K_{p^\alpha}},\widehat{p^\alpha})$ and $G_{j}:=G_{j-1} \overleftarrow{\nabla}_{\widehat{p^{\alpha-j}}}(\overleftrightarrow{K_{p^{\alpha+j-1}(p-1)}},\widehat{p^{\alpha-j}})$ for $1\leq j$ (see Figure \ref{fig:quiver}).
\end{thm}

In short terms, the full coloring quiver $\mathcal{Q}_{\dihprime}(\TB\left(\frac{N}{M}\right))$ has its vertex set partitioned into orbits of $[0,p^j]$ for $\alpha-\beta \leq j\leq \alpha$, each of which induces a regular complete subgraph, and has $p^i$ directed edges from each vertex from the orbit of $[0,p^i]$ to each vertex from the orbit of $[0,p^j]$ whenever $i\leq j$. If the order of the dihedral quandle is fixed, then the number $\beta$ determines the number of components of the quiver.

\begin{figure}[ht]
    \centering
    \includegraphics[width=\textwidth]{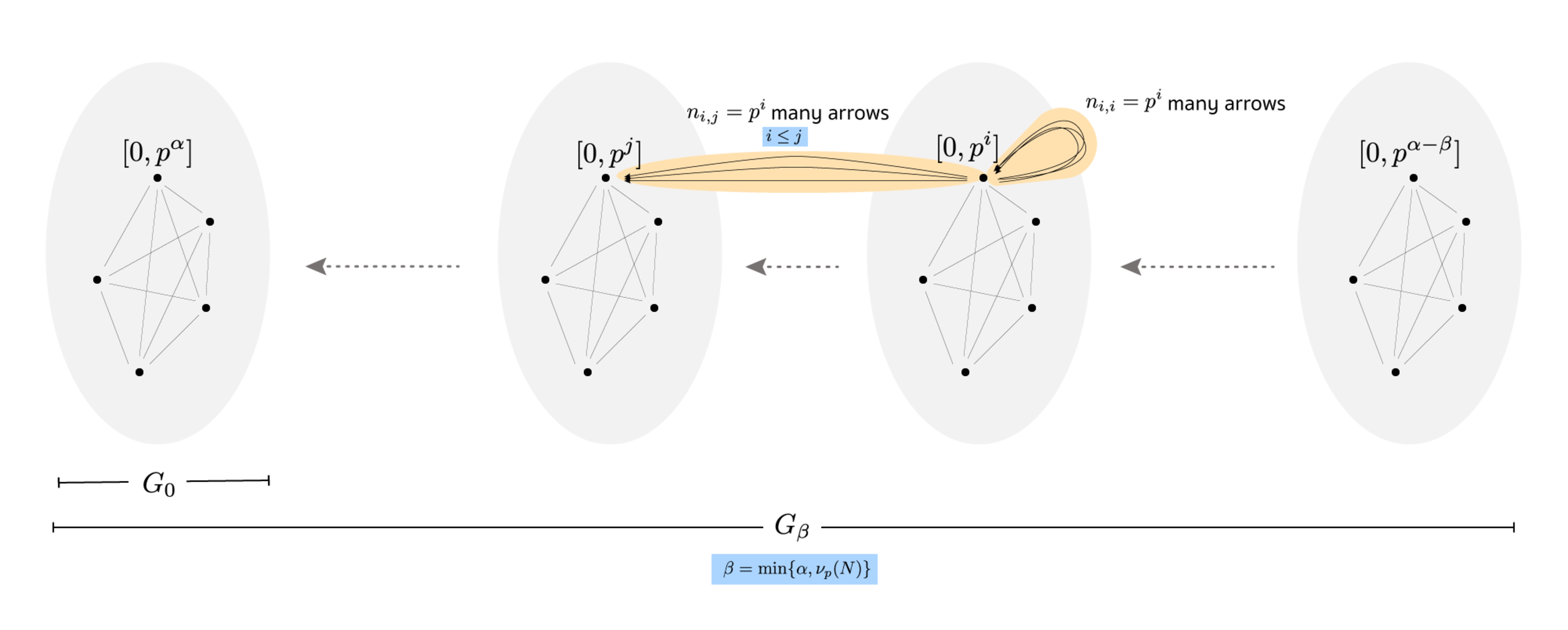}
    \caption{\label{fig:quiver} The full coloring quiver $\mathcal{Q}_{\dihprime}(\TB\left(\frac{N}{M}\right))$.}
\end{figure}

For instance, suppose that $L$ is the 4-crossing torus link and our quandle is $\mathbb{Z}_4^{dih}$. Then, $\{[0,0],[1,1],[2,2],[3,3]\}$ constitutes an orbit,\\ $\{[0,1],[1,2],[2,3],[3,0],[0,3],[1,0],[2,1],[3,2]\}$ constitutes an orbit, and \\ $\{[0,2],[1,3],[2,0],[3,1]\}$ constitutes an orbit.

\begin{cor} Let $p$ be a prime and $N,M\in\mathbb{N}$ with $\gcd(N,M)=1$. Then, the quiver $$\QQ_{\mathbb{Z}_{p}^{dih}}\left(\TB\left(\frac{N}{M}\right)\right)\cong \begin{cases}(\overleftrightarrow{K_{p}},\widehat{p})\overleftarrow{\nabla}_{\widehat{1}}(\overleftrightarrow{K_{p(p-1)}},\widehat{1}) & \text{ if }p\mid N,\\ (\overleftrightarrow{K_{p}},\widehat{p})&  \text{ if }p\nmid N.\end{cases}$$ \end{cor}

\begin{proof}Set $\alpha=1$ in theorem \ref{mainresult}. \end{proof}

\subsection{The general case}

For convenience, we start using multi-index notation. For a fix positive integer $n$, we write the prime decomposition $n=\prod_i p_i^{\alpha_i}$ as ${p}^{\alpha}$, where $p$ is regarded as the sequence of distinct prime factors and $\alpha$ is regarded as the sequence of corresponding exponents. For sequences of nonnegative integers $j=(j_i)$ and $j'=(j_i')$  with the same length as $p$, we write ${p^j}:= \prod_i p_i^{j_i}$, and define ${j}\preceq{j}'$ iff $j_i\leq j_i'$ for all $i$.

The next result generalizes Lemma \ref{orbitdecomp}. In a similar manner, we set the sequence $\beta$ with $\beta_i=\min\{\alpha_i,\nu_{p_i}(N)\}$.

\begin{lem} \label{orbitdecomp2} 
Under the action of $\Aut(\dih)$ on $\Hom(Q\left(\TB\left(\frac{N}{M}\right)\right), \dih)$ , for $\alpha-\beta \preceq j,j'\preceq \alpha$ we have
\begin{enumerate}
\item $[0, p^j]\in \Hom(Q\left(\TB\left(\frac{N}{M}\right)\right), \dih)$.
\item $[0,p^{j}]$ and $[0,p^{j'}]$ lie in same orbit if and only if $j=j'$.
\item The size of the orbit of $[0,p^j]$ is given by $n_j :=\prod_i n_{j_i}$, where $$n_{j_i}=\begin{cases} p_i^{2\alpha_i-j_i-1}(p_i-1) &\text{ if } j_i<\alpha_i, \\ p_i^{\alpha_i} &\text{ if }j_i =\alpha_i. \end{cases}$$
\item $\Hom(Q\left(\TB\left(\frac{N}{M}\right)\right), \dih)$ is partitioned into orbits with $\{[0,p^j]: \alpha-\beta \preceq j\preceq \alpha\}$ being a complete set of representatives.
\item The number $n_{j,j'}$ of endomorphisms of $\dih$  sending $[0,p^j]$ to $[0,p^{j'}]$ is given by $$n_{j,j'}=\begin{cases} 0 & \text{ if } j\not\preceq j',\\ p^j & \text{ if }j\preceq j'.
\end{cases}$$
\end{enumerate}
\end{lem}

\begin{proof} The proof also closely follows the proof of Lemma \ref{orbitdecomp}
\begin{enumerate}
\item For each $i$, we have $\alpha_i\leq \nu_{p_i}(N)+j_i$ since $\alpha_i-\nu_{p_i}(N)\leq \alpha_i-\beta_i \leq j_i$. Thus, $p^\alpha \mid Np^j$ and $[0,p^j]\in \Hom(Q\left(\TB\left(\frac{N}{M}\right)\right), \dih)$.

\item The converse is obvious. Without loss of generality, suppose that $j_i<j_i'$ for some index $i$. Suppose for contradiction that there is $\llbracket x,y\rrbracket \in \Aut(\dih)$ such that $[x,p^j(y-x)+x]=\llbracket x,y\rrbracket\tr [0,p^j]=[0,p^{j'}]$. We see that $x=0$ and $p^j y\equiv p^{j'}\pmod {p^{\alpha}}$. This implies $p_i \mid y$ and $\gcd(y,n)\geq p_i>1$,  which contradicts with $y \in \mathbb{Z}_n^\times$. Hence, $[0,p^j]$ and $[0,p^{j'}]$ lie in different orbits.

\item We also try to determine the size of the stabilizer of $[0,p^j]$, which is equal to the number of $ \llbracket x,y\rrbracket \in\Aut(\dih)$ such that $ \llbracket x,y\rrbracket \tr [0,p^j]=[0,p^j]$. We see that $x=0$ and $y\in\mathbb{Z}_n^\times$ satisfying $p^jy\equiv p^j \pmod {n}$. By looking at each prime, the condition is equivalent to the solving the system $p^{j_i} y\equiv p^{j_i} \pmod{p_i^{\alpha_i}}$ with $\gcd(y,p_i^{\alpha_i})=1 $ for each $i$.
\begin{description}
\item [Case 1] $j_i=\alpha_i$.  The condition $p^{j_i} y\equiv p^{j_i} \pmod{p_i^{\alpha_i}}$ is trivial, so there are $\phi(p_i^{\alpha_i})$  solutions.
\item [Case 2] $j_i<\alpha_i$. In this case, there are $p_i^{j_i}$ solutions of the form $y = 1+ kp_i^{\alpha_i-j_i}\pmod{p_i^{\alpha_i}}$, where $0\leq k<p_i^{j_i}$. Note that the solutions satisfy $\gcd(y,p_i^{\alpha_i})=1$.
\end{description}

Let us define $$m_{j_i}=\begin{cases} p_i^{j_i} &\text{ if } j_i<\alpha_i, \\ \phi(p_i^{\alpha_i})&\text{ if }j_i =\alpha_i. \end{cases}$$
By Chinese remainder theorem, the size of the stabilizer of $[0,p^j]$ is $\prod_i m_{j_i}$. Thus, by orbit-stabilizer theorem, the size of the orbit of $[0,p^j]$ is $$\frac{n\phi(n)}{\prod_i m_{j_i}}=\prod_i \frac{p_i^{\alpha_i} \phi(p_i^{\alpha_i})}{m_{j_i}}=\prod_i \frac{p_i^{2\alpha_i-1}(p_i-1)}{m_{j_i}}=\prod_i n_{j_i}.$$ 

\item Consider the total of the size of orbits \begin{align*}
\sum_{\alpha-\beta \preceq j \preceq \alpha} n_j &= \sum_{\alpha_I-\beta_I \leq j_I \leq \alpha_I}\dots \sum_{\alpha_1-\beta_1 \leq j_1 \leq \alpha_1} \prod_i n_{j_i} \\
&= \prod_{i} \sum_{\alpha_i-\beta_i \leq j_i \leq \alpha_i} n_{j_i} \\
&= \prod_{i} \left[p_i^{\alpha_i} + \sum_{\alpha_i-\beta_i \leq j_i < \alpha_i}p_i^{2\alpha_i-j_i-1}(p_i-1) \right]\\
&= \prod_{i} \left[ p_i^{\alpha_i} + p_i^{2\alpha_i-1}(p_i-1)\cdot \frac{1}{p_i^{\alpha_i-\beta_i}} \sum_{0\leq j_i< \beta_i} \frac{1}{p_i^{j_i}}\right]\\
&=\prod_{i} \left[ p_i^{\alpha_i} + p_i^{2\alpha_i-1}(p_i-1)\cdot \frac{1}{p_i^{\alpha_i-\beta_i}} \cdot\frac{1-\frac{1}{p_i^{\beta_i}}}{1-\frac{1}{p_i}}\right]\\
&=\prod_{i}  p_i^{\alpha_i+\beta_i} = p^{\alpha+\beta}.
\end{align*} 
Since $|\Hom(Q\left(\TB\left(\frac{N}{M}\right)\right), \dih)|=n\gcd(N,n)=p^{\alpha+\beta}$, we have all elements from these orbits. 

\item  This also follows from Proposition \ref{countendo}.   
\end{enumerate}  \end{proof}

Combining all the results from Lemma \ref{numedgeequal} and Lemma \ref{orbitdecomp2}, we are able to determine the full quandle coloring quiver of the two-bridge link $\TB\left(\frac{N}{M}\right)$ with respect to the quandle $\dih$.

\begin{notn} Let $\Lambda$ be a set, $G=\{G_\lambda\}_{\lambda\in \Lambda}$ be a family of graphs indexed by $\Lambda$, and $w:\Lambda\times \Lambda \to \mathbb{N}_0$ be a map. Denote by $\overleftarrow\nabla_w G$ the disjoint union graph $\bigsqcup_{\lambda\in\Lambda}G_\lambda$ with additional $w(\lambda,\mu)$ directed edges from each vertex of $G_\lambda$ to each vertex of $G_\mu$. With this notion, $G_2\overleftarrow\nabla_{\hat{m}} G_1 = \overleftarrow\nabla_w \{G_1,G_2\},$ where $w:\{1,2\}\times \{1,2\}\to \mathbb{N}_0$ is given by $w(1,2)=m$ and $w(2,1)=w(1,1)=w(2,2)=0$.\end{notn}

\begin{thm} Let $n$ be a positive integer and write $n=\prod_i p_i^{\alpha_i}$, where $p_i$ are distinct primes and $\alpha_i > 0$. Let $N,M$ be positive integers with $\gcd(N,M)=1$ and set $\beta_i=\min\{\alpha_i,\nu_{p_i}(N)\}$. Let $\Lambda=\{j: \alpha-\beta\preceq j\preceq \alpha\}$. The full quandle coloring quiver of the two-bridge link $\TB\left(\frac{N}{M}\right)$ with respect to the quandle $\dih$ is given by $$\QQ_{\dih}\left(\TB\left(\frac{N}{M}\right)\right)\cong \overleftarrow\nabla_w \{(\overleftrightarrow{K_{n_j}},\widehat{p^j}):j\in\Lambda\},$$ where $w:\Lambda\times \Lambda \to \mathbb{N}_0$ is given by $$w(j,j')=\begin{cases} p^j & \text{ if }j\preceq j'\text{ and }j\neq j', \\ 0 &\text{ else}.\end{cases}$$ \end{thm}

The full quandle coloring quiver $\QQ_{\dih}\left(\TB\left(\frac{N}{M}\right)\right)$ is a higher dimensional generalization of that when $n$ is a prime power. Its vertex set is partitioned into orbits that can be arranged into a higher dimensional grid with width in the $i$-th dimension depending only on $\beta_i$.  We can see in the proof of Lemma \ref{orbitdecomp2} that problems reduce to subproblems for each prime dividing the order of the dihedral quandle. Roughly speaking, the orbits and stabilizers split into "products". See section 4 of \cite{taniguchi2021quandle} for more rigorous discussion of this situation.

\begin{exmp}
    The torus link $\TT(N,2)\cong \TB\left(\frac{N}{1}\right)$. The full quandle coloring quiver $\QQ_{\mathbb{Z}_{12}^{dih}}(\TT(36,2))$ is shown in Figure \ref{fig:quiver3}.
\end{exmp}

\begin{figure}[ht]
    \centering
    \includegraphics[width=0.8\textwidth]{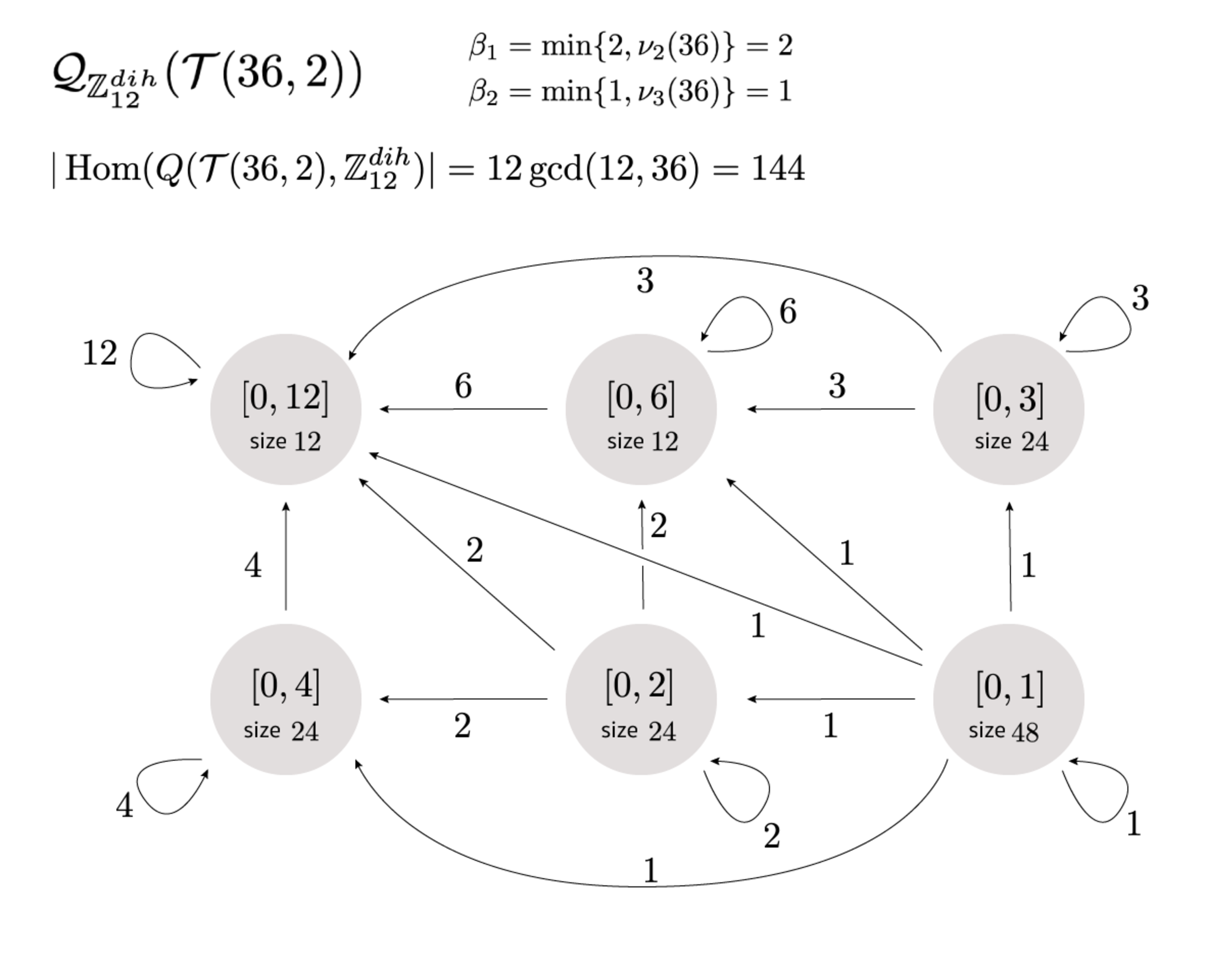}
    \caption{\label{fig:quiver3} The full quandle coloring quiver $\QQ_{\mathbb{Z}_{12}^{dih}}(\TT(36,2))$.}
\end{figure}

\subsection{Applications and remarks}
The formulas of quandle cocycle invariants of 2-bridge links are given in \cite{iwakiri2005calculation,carter2003quandle} for dihedral quandles of prime orders. This information can be combined with our results to calculate the quandle cocycle quivers of 2-bridge links \cite{nelson2019quandle}. Similarly, the authors of \cite{scott2005cocycle} computed quandle module invariants using some dihedral quandles, which can be used to compute the quandle module quivers \cite{istanbouli2020quandle} when combined with our result.

By a result of Taniguchi \cite{taniguchi2021quandle}, the quandle coloring quiver is not a stronger invariant if one uses the dihedral quandle of order $n=p_1p_2\cdots p_k$, where $p_i$ is a prime number. To find an instance of proper enhancement, we may have to consider a quandle whose order is a power of a prime.
\begin{exmp}

Consider the dihedral quandle $Q=\mathbb{Z}_4^{dih}$. Then, the quandle coloring number of $T(9,3)$ and $T(4,2)$ by $Q$ are both 16. By the main result of this paper and a result in \cite{zhou2023quandle}, the associated quiver invariants are not equal. In particular, the quiver for $T(4,2)$ contains three complete graphs $\overleftrightarrow{K_{4}}, \overleftrightarrow{K_{4}}$, and $\overleftrightarrow{K_{8}}$. On the other hand, the quiver for $T(9,3)$ contains four copies of complete graphs that are all $\overleftrightarrow{K_{4}}$ as shown schematically in Figure 6 of \cite{zhou2023quandle} (merging parallel edges). More examples can be obtained by replacing $9$ with $6k+3$ where $k=1,2,3,...$
    
\end{exmp}

Of course, other invariants already distinguish the links in the examples above, but our computations offer additional tools for potential use in the future to distinguish unknown knotted objects.

\subsection*{Acknowledgments}
 The research conducted for this paper is supported by the Pacific Institute for the Mathematical Sciences (PIMS). The first author is supported by the Centre of Excellence
in Mathematics, the Commission on Higher Education, Thailand. The research and findings may not reflect those of the Institute. The third author thanks Nicholas Cazet for helpful conversations and for introducing him to Fielder's work. We are grateful to Chris Soteros for support.
\bibliographystyle{plain}
\bibliography{refs}

\end{document}